\documentclass[11pt, a4paper]{amsart}

\usepackage{color}
\usepackage{amsfonts}
\usepackage{amssymb}
\usepackage{amsmath}
\usepackage{hyperref}

\newcommand{\ds}{\displaystyle}

\newcommand{\cii}[1]{_{ {}_{ #1}}}
\newcommand{\eq}[2][label]{\begin{equation}\label{#1}#2\end{equation}}

\newcommand{\av}[2]{\langle #1\rangle\cii {#2}}
\newcommand{\df}{\buildrel\rm{def}\over=}

\newcommand{\const}{{\rm const}}
\newcommand{\Bel}{\mathbf B}

\newcommand{\ma}{Monge--Amp\`{e}re }

\newcommand{\bp}[1]{\mathfrak{b}_{#1}}

\newcommand{\pd}{\partial}

\def\supp{\operatorname{supp}}

\newcommand{\ch}{\operatorname{ch}}

\newcommand{\Qq}{[w]_{A_1}}

\newcommand{\eps}{\varepsilon}
\newcommand{\la}{\lambda}
\newcommand{\vf}{\varphi}

\newcommand{\cB}{\mathcal B}
\newcommand{\cD}{\mathcal D}

\newcommand{\bR}{\mathbb R}

\newcommand{\bZ}{\mathbb Z}

\newcommand{\mfA}{\mathfrak A}

\newcommand{\bfM}{\mathbf{M}}

\newtheorem{theorem}{Theorem}[section]
\newtheorem{cor}[theorem]{Corollary}
\newtheorem{lemma}[theorem]{Lemma}
\newtheorem{prop}[theorem]{Proposition}
\theoremstyle{definition}
\newtheorem{remark}[theorem]{Remark}
\newtheorem{defin}[theorem]{Definition}

\numberwithin{equation}{section}

\newcounter{vremennyj}

\begin{document}
\title{On weak  weighted estimates of martingale transform}
\author[F.~Nazarov]{Fedor Nazarov}
\thanks{FN is partially supported by the NSF grant DMS-1265623}
\address{Department of Mathematics, Kent Sate University, USA}
\email{nazarov@math.kent.edu \textrm{(F. \ Nazarov)}}
\author[A.~Reznikov]{Alexander Reznikov}
\address{Department of Mathematics, Michigan Sate University, East Lansing, MI. 48823}
\email{reznikov@math.msu.edu \textrm{(A. \ Reznikov)}}
\author[V.~Vasyunin]{Vasily Vasyunin}
\address{V. A. Steklov Math. Inst., Fontanka 27, St. Petersburg, Russia}
\email{vasyunin@pdmi.ras.ru \textrm{(V.\ Vasyunin)}}
\author[A.~Volberg]{Alexander Volberg}
\thanks{AV and VV are partially supported by the program ``Research in Pairs'' of the Oberwolfach Institute 
for Mathematics, and by the NSF grant DMS-1600065; 
AV is also supported by the NSF grant DMS-1265549, and VV is also supported by the RFBR grant 14-01-00748.}
\address{Department of Mathematics, Michigan Sate University, East Lansing, MI. 48823}
\email{volberg@math.msu.edu \textrm{(A.\ Volberg)}}
\makeatletter
\@namedef{subjclassname@2010}{
  \textup{2010} Mathematics Subject Classification}
\makeatother
\subjclass[2010]{42B20, 42B35, 47A30}
\keywords{}
\begin{abstract} 
We consider several weak type estimates for  singular operators using the Bellman function approach. We disprove the $A_1$ conjecture, which stayed open after  Muckenhoupt--Wheeden's conjecture was  disproved by Reguera--Thiele.
\end{abstract}
\maketitle

\section{End-point estimates. Notation and facts.}
\label{introEndPoint}

The end-point estimates play an important part in the theory of singular integrals 
(weighted or unweighted). They are usually the most difficult estimates in the theory, 
and the most interesting of course. 
It is a general principle that one can extrapolate the estimate from the end-point 
situation to all other situations.  We refer the reader to the book \cite{CUMPbook}
that treats this subject of extrapolation in depth. 

On the other hand, it happens quite often that the singular integral estimates exhibit 
a certain ``blow-up'' near the end point. To catch this blow-up can be a difficult task. 
We demonstrate this hunt for  blow-ups by the examples of weighted dyadic singular 
integrals and their behavior in $L^p(w)$.
The end-point $p$ will be naturally $1$ (and  sometimes slightly unnaturally $2$) 
depending on the martingale  singular operator. The singular integrals in this article 
are the easiest possible. They are dyadic martingale operators on $\sigma$-algebra 
generated by usual homogeneous dyadic lattice on the real line. We do not consider 
any non-homogeneous situation, and this standard $\sigma$-algebra generated by a dyadic 
lattice $\cD$ will be provided with Lebesgue measure. 

Our goal will be to show how the technique of Bellman function gives 
the proof of the blow-up of the weighted estimates of the corresponding 
weighted dyadic singular operators. This blow-up will be demonstrated 
by certain estimates from below of the Bellman function of a dyadic problem. 
Interestingly, one can bootstrap then the correct estimates from below of 
a dyadic operators to the estimate from below of such classical operators as e. g.  
the Hilbert transform. The same rate of blow-up then persists for the classical operators. 
But this bootstrapping argument will be carried out in a separate note, here, 
for simplicity, we work only with dyadic martingale operators.

As to the Bellman function part of our consideration below, this part will be reduced to 
the task to find the lower estimate for the solutions of the homogeneous Monge--Amp\`ere 
differential equation.

\subsection{End-point estimates for martingale transform}
\label{MartAndSq}

We will work with a standard dyadic filtration $\cD=\cup_k \cD_k$ on $\bR$,
where 
$$
\cD_k\df\Big\{[\frac m{2^k},\frac{m+1}{2^k})\colon m\in\bZ\Big\}\,.
$$
We consider the martingale transform related to this homogeneous dyadic filtration. 

The symbol $\av{\vf}I$ denotes average value of $\vf$ over the set $I$ 
i.\,e., 
$$
\av{\vf}I=\frac1{|I|}\int_I\vf(t)\,dt.
$$
We consider martingale differences 
(recall that the symbol $\ch(J)$ denotes the dyadic children of $J$)
$$
\Delta\cii J\vf\df\sum _{I\in\ch(J)} \chi\cii I(\,\av{\vf}I\!-\av{\vf}J).
$$

For our case of dyadic lattice on the line we have  that $|\Delta_J\vf|$ is constant on $J$, 
the set $\ch(J)$ consists of two halves of $J$ ($J^+$ and $J^-$), and
$$
\Delta\cii J\vf =\frac12(\,\av\vf{J^+}\!-\av\vf{J^-}\!)(\,\chi\cii{J^+}\!-\chi\cii{J^-}\!)\,.
$$

We consider the dyadic $A_1$ class of weights, but we skip the word dyadic in 
what follows, because we consider here only dyadic operators.
A positive  function $w$ is called an $A_1$ weight if
$$
\Qq\df\sup_{J\in \cD}\frac{\av wJ}{\inf_J w}<\infty\,.
$$

By $Mw$ we will denote the dyadic maximal function of $w$, that is 
$Mw(x)=\sup\{\av wJ\colon J\in\cD,\,J\ni x\}$. Then $w\in A_1$ with ``norm'' $Q$ means that
$$
Mw \le Q \cdot w\quad a.\,e.\,,
$$
and $Q=\Qq$ is the best constant in this inequality.

Recall that a martingale transform is an operator given by
$
T_\eps\vf =\sum_{J\in\cD}\eps\cii J\Delta\cii J\vf\,.
$
It is convenient to use Haar function $h\cii J$ associated with dyadic interval $J$,
$$
h\cii J(x):=\begin{cases}
\phantom{-}\frac1{|J|^{1/2}},& x\in J^+\,;
\\
-\frac1{|J|^{1/2}},& x\in J^-\,,
\\
\hskip10pt 0,& x\notin J\,.
\end{cases}
$$
Sometimes it is more convenient to use the Haar functions $H\cii J$ normalized in $L^\infty$: 
$H\cii J=|J|^{1/2}h\cii J$.
In this notations, the martingale transform $\psi$ of a function $\vf$ is
$$
\psi=T_\eps\vf=\sum_{J\in\cD}\frac{\eps\cii J}{|J|}(\vf, H\cii J)H\cii J=
\sum_{J\in\cD}\eps\cii J(\vf, h\cii J)h\cii J\,,
$$
where $(\,\cdot\,,\,\cdot\,)$ stands for the scalar product in $L_2$.
In all our calculations we always think the sum has only unspecified but finite number of terms, 
so we may not to worry about the converges of this series. Nevertheless approximation arguments
give us the final estimates for an arbitrary $L^1$~function~$\vf$.
As to the values of the multiplicator coefficients we consider the class $|\eps\cii J|\le1$
or its important subclass $\eps\cii J=\pm1$.

We are interested in the weak estimate for the martingale transform $T$ 
in the weighted space $L^1(\bR, w\,dx)$, where $w\in A_1$. 
The end-point exponent is naturally $p=1$, and we wish to understand the 
order of magnitude of the constant $C(\Qq)$ in the weak type inequality for the 
dyadic martingale transform:
\begin{equation}
\label{weakMT}
\frac1{|I|}\sup_\eps w\{t\in I\colon\!\sum_{J\in \cD(I)}\!\eps\cii J (\vf, h\cii J) h\cii J(t)\ge\la\} 
\le C(\Qq)\frac{\av{|\vf|w}I}{\la}\,.
\end{equation}
Here $\vf$ runs over all functions  such that $\supp \vf\subset I$ and $\vf\in L^1(I,w\,dt)$, 
$w\in A_1$; $\eps=\{\eps\cii J\}$ and $|\eps\cii J|\le1$. For a set $S$ we write $w(S)$ for $\int_Sw(t)\,dt$. 
This paper is devoted to the study of the ``sharp'' order of magnitude of constants 
$C(\Qq)$ in terms of $\Qq$ if $\Qq$ is large. We are primarily interested in the estimate of 
$C(\Qq)$ from below, that is in finding the worst possible $A_1$ weight in terms of weak type 
estimate (of course this involved also finding the worst test function $\vf$ as well).
 
We will prove the following result.
\begin{theorem}
\label{weakMT-t}
For any $Q$\textup, $Q\ge4$\textup, there is a weight $w\in A_1$ with $\Qq=Q$ such that constant $C(Q)$ 
from \eqref{weakMT} satisfies
$$
C(Q)\ge\frac1{515}Q (\log Q)^{1/3}.
$$
\end{theorem}

In~\cite{LOP} the following estimate from above has been proved:

\begin{theorem}
\label{LOP-t}
There is a positive absolute constant $c$ such that for any weight $w\in A_1$ 
estimate~\eqref{weakMT} holds with
$$
C(\,\Qq)=c\,\Qq\log \Qq\,.
$$
\end{theorem}

\begin{remark}
The sharp  power remains enigmatic.
\end{remark}

\subsection{Two problems of Muckenhoupt}
Theorem \ref{weakMT-t} is a subtle result and it will take some space below to prove. Recall that Muckenhoupt 
conjectured that for the Hilbert transform $H$ and any weight $w\in A_1$ the following 
two estimates hold on a unit interval $I$:
\begin{equation}
\label{MstrongConj}
w\{x\in I: |Hf(x)| >\la\} \le \frac{C}{\la} \int_I |f| Mw dx\,,
\end{equation}
\begin{equation}
\label{MweakConj}
w\{x\in I: |Hf(x)| >\la\} \le \frac{C\, \Qq}{\la} \int_I |f| w dx\,.
\end{equation}

Obviously if \eqref{MstrongConj} holds then \eqref{MweakConj} is valid as well.
It took many years to {\it disprove} \eqref{MstrongConj}. This was done by Maria 
Reguera and Christoph Thiele \cite{R} (for the martingale transform), \cite{RT} (for the Hilbert transform). The constructions involve a very 
irregular (almost a sum of delta measures) weight $w$, so there was a hope that such 
an effect cannot appear when the weight is regular in the sense that $w\in A_1$. 
Theorem~\ref{weakMT-t} gives a counterexample to this hope for the case when the Hilbert 
transform is replaced by the martingale transform on a usual homogeneous dyadic filtration.  
The reader can consult~\cite{NRVV} to see that for the Hilbert transform a counterexample 
also exists, and so~\eqref{MweakConj} fails as well. The counterexample  for the Hilbert 
transform is the transference of a counterexample we build here
for the martingale transform.  It will be published separately. The blow-up estimate $Q(\log Q)^{1/3}$ holds for the Hilbert transform as well.

Theorem~\ref{logQthm} below implicitly gives a certain counterexample for the Hilbert 
transform, but it takes some work to see that. We will explain in a separate note how to make this transference.

\subsection{Plan of the paper} The main results and main difficulties in getting them are in Section \ref{PropBEL-MT}. 
This is where the weighted problem is considered and $A_1$  question of Muckenhoupt is answered in 
the negative. But we start with unweighted case, where the weight $w$ is just identically $1$. 
This is done in Section \ref{MT} below. It has its own interest because of unexpected Corollary \ref{const1}.
But also it serves the goal by preparing the reader to the more sophisticated reasoning
 in Section \ref{PropBEL-MT}. It also has the following feature: in 
 Section \ref{PropBEL-MT} we are unable to find exactly the corresponding Bellman function, 
 we just managed to estimate this difficult object. But in a simpler, unweighted case of Section 
 \ref{MT} we find the formula  for the corresponding simpler Bellman function. Incidentally
  it serves as a boundary value function for the weighted Bellman function of Section \ref{PropBEL-MT} whose precise formula eludes us.

\bigskip
\section{Unweighted estimate of the martingale transform}
\label{MT}

In this Section we prove the following unweighted analog of inequality~\eqref{weakMT}
\eq[noW-weakMT]{
\frac1{|I|}\big|\{t\in I\colon\sum_{J\in\cD(I)}\!\eps\cii J(\vf,h\cii J)h\cii J(t)\ge\la\}\big| 
\le 2\frac{\av{|\vf|}I}{\la}\,.
}

We will work not on the whole $\bR$ but on a finite interval. The result for the whole axis
can be obtain by enlarging the underlying interval and the fact that the estimates will not depend
on the interval. So, we are working on $I=[0,1]$. The symbol $\cD=\cD(I)$ means the dyadic lattice
of subintervals. Let $\vf$ be a dyadic martingale starting at $x_1$ and $\psi$ is its martingale 
transform, starting at $x_2$, i.\,e.,
$$
\vf=x_1+\!\sum_{J\in\cD(I)}\!\!(\vf, h\cii J) h\cii J\,, \qquad
\psi=x_2+\!\sum_{J\in\cD(I)}\!\!\eps\cii J(\vf,h\cii J)h\cii J\,.
$$
We consider two classes of martingale transforms: 1) the case of $\pm$-trans\-forms, i.\,e. 
the case when we assume that $\eps\cii J=\pm1$; and 2) the case when the martingale $\psi$
is differentially subordinate to $\vf$, i.\,e. the case when we assume that $|\eps\cii J|\le1$.
The first class of admissible pairs $\{\vf,\,\psi\}$ we denote by $\mfA_\pm$, the second one by
$\mfA_\eps$.
 
The desired estimate we deduce to estimating a certain function of three variables
related to our inequality, which is called the Bellman function of the problem. In fact the
Bellman function related to some inequality is simply the extremal value of the quantity we
need to estimate under several fixed parameters related to the problem. Describe now the
Bellman function of our problem.

With every pair of functions $\{\vf,\,\psi\}$ on $I$ we associate the so called Bellman point
$\bp{\vf,\psi}=x=(x_1,x_2,x_3)$ with coordinates
$$
x_1=\av\vf{I},\qquad x_2=\av\psi{I},\qquad x_3=\av{|\vf|}I.
$$
The set of all admissible pairs corresponding to a point $x$ will be denoted by $\mfA_\pm(x)$
in the case of $\pm$transform and by $\mfA_\eps(x)$ in the case of differential subordination.
Our Bellman function is the following one:
\begin{equation}
\label{def_B}
\Bel(x)=\Bel(x_1,x_2,x_3):= \sup_{\mfA(x)}\,\frac1{|I|}\big|\{t\in I\colon
\!\!\sum_{J\in\cD(I)}\!\!\psi(t)\ge0\}\big|\,,
\end{equation}
where $\mfA(x)$ is either $\mfA_\pm(x)$ or $\mfA_\eps(x)$.
If we would like to specify that we speak about $\pm$-transform, i.\,e. supremum is
taken over the set $\mfA=\mfA_\pm$, then the corresponding Bellman function will be
written as $\Bel_\pm$, and we shall write $\Bel_\eps$ if $\mfA=\mfA_\eps$. This index
will be omitted in any assertion valid in both cases. It is clear, that $\Bel_\pm\le\Bel_\eps$,
but as we will see at the end these two functions coincide.
Note that the function $\Bel$ should not be indexed by $I$ because it is easy to check that this
function does not depend on $I$.

\subsection{Properties of $\Bel$}
\label{properties}

\subsubsection{Domain and Range} 
Formally the definition of $\Bel$ is correct for arbitrary $x\in\bR^3$, but there is no sense
to consider $\Bel$ at the points where the set of admissible functions is empty, and therefore
the corresponding supremum is $-\infty$. We would like to consider the function $\Bel$ on
the domain $\Omega\subset \bR^3$: 
$$
\Omega\df\{x=(x_1,x_2,x_3)\in\bR^3: |x_1|\le x_3\}\,.
$$

For any $x\in\Omega$, the set of test functions $\mfA(x)$ is not empty and it is immediately clear from 
the definition that
$$
0\le\Bel(x)\le1\,.
$$

\subsubsection{Symmetry} The function $\Bel$ is invariant under reflection with
respect $x_1$:
$$
\Bel(-x_1,x_2,x_3)=\Bel(x_1,x_2,x_3)\,,
$$
because if $\bp{\vf,\psi}=(x_1,x_2,x_3)$, then $\bp{-\vf,\psi}=(-x_1,x_2,x_3)$.

\subsubsection{Homogeneity}
$$
\Bel(\tau x_1,\tau x_2,\tau x_3)=\Bel(x_1,x_2,x_3)\,,\qquad\tau>0\,,
$$
because if $\bp{\vf,\psi}=(x_1,x_2,x_3)$, then $\bp{\tau\vf,\tau\psi}=(\tau x_1,\tau x_2,\tau x_3)$, and the functions
$\psi$ and $\tau\psi$ are positive simultaneously.

\subsubsection{Boundary condition}
\begin{equation}
\label{bc}
\Bel(0,x_2,0)=
\begin{cases}
1\,,&\text{if }\;x_2\ge0\,,
\\
0\,,&\text{if }\;x_2<0\,,
\end{cases}
\end{equation}
because the only admissible pair for the point $x=(0,x_2,0)$ is $\vf=0$,
$\psi=x_2$.

\subsubsection{Obstacle condition}
\begin{equation}
\label{obst_cond}
\Bel(x_1,x_2,|x_1|)\ge
\begin{cases}
1\,,&\text{if }\;x_2\ge0\,,
\\
0\,,&\text{if }\;x_2<0\,,
\end{cases}
\end{equation}
because the pair of constant functions $\vf=x_1$, $\psi=x_2$ is an admissible pair for the 
point $x=(x_1,x_2,|x_1|)$.

By the way, since $\Bel\le1$ by the definition, the obstacle condition supplies us with
the function on the half of the boundary, namely, $\Bel(x)=1$ if $x$ is on the boundary and 
$x_2\ge0$. We shall see soon that this is not the whole part of the boundary where $\Bel(x)=1$. 
However first we derive the main inequality.

\subsection{Main inequality}
\label{MI}

\begin{lemma}
\label{MTtuda} 
Let $x^\pm$ be two points in $\Omega$ such that
\begin{itemize}
\item
$|x^+_2-x^-_2|=|x^+_1-x^-_1|$ in the case $\Bel=\Bel_\pm$\textup;
\item
$|x^+_2-x^-_2|\le|x^+_1-x^-_1|$ in the case $\Bel=\Bel_\eps$\textup,
\end{itemize}
and $x=\frac12(x^++x^-)$. Then
\begin{equation}
\label{mi} \Bel(x)-\frac{\Bel(x^+)+\Bel(x^-)}2\ge 0\,.
\end{equation}
\end{lemma}

\begin{proof}
Fix $x^\pm\in\Omega$, and let $\vf^\pm,\;\psi^\pm$ be two pairs of test functions giving the supremum
in $\Bel(x^+)$, $\Bel(x^-)$ respectively up to a small number $\eta>0$. Using the fact that the
function $\Bel$ does not depend on the interval where the test functions are defined, we 
assume that $\vf^+$, $\psi^+$ are supported on $I^+$ and $\vf^-$, $\psi^-$  are on $I^-$,
where $I^\pm$ are two halves of the interval $I$:
$$
\vf^\pm=x^\pm_1+\!\!\!\!\sum_{J\in\cD(I^\pm)}\!\!\!\!a\cii J h\cii J\,, \qquad 
\psi^\pm=x^\pm_2+\!\!\!\!\sum_{J\in\cD(I^\pm)}
\!\!\!\!\eps\cii J a\cii Jh\cii J\,.
$$
And we assume that for these functions the estimates
$$
\frac1{|I^\pm|}\big|\{t\in I^\pm\colon \psi^\pm(t)\ge0\}\big|
\ge\Bel(x^\pm)-\eta
$$
hold. Consider the functions
$$
\vf(t):=
\begin{cases}
\vf^+(t)\,,&\text{if }\;t\in I^+
\\
\vf^-(t)\,,&\text{if }\;t\in I^-
\end{cases}
=\frac{x^+_1+x^-_1}2+\frac{x^+_1-x^-_1}2h\cii I+ 
\!\!\!\!\sum_{J\in\cD(I)}\!\!\!\!a\cii Jh\cii J
$$
and
$$
\psi(t):=
\begin{cases}
\psi^+(t)\,,&\text{if }\;t\in I^+
\\
\psi^-(t)\,,&\text{if }\;t\in I^-
\end{cases}
=\frac{x^+_2+x^-_2}2+\frac{x^+_2-x^-_2}2h\cii I+
\!\!\!\!\sum_{J\in\cD(I)}\!\!\!\!\eps\cii Ja\cii Jh\cii J\,.
$$
Under our assumption about relation between $|x^+_1-x^-_1|$ and $|x^+_2-x^-_2|$ we have 
$\{\vf,\psi\}\in\mfA_\pm$ in the first case and $\{\vf,\psi\}\in\mfA_\eps$ in the second one,
i.\,e. in each case this is an admissible pair of the test
functions corresponding to the point $x$. Therefore,
\begin{align*}
\Bel(x)&\ge\frac1{|I|}\big|\{t\in I\colon \psi(t)\ge0\}\big|
\\
&=\frac1{2|I^+|}\big|\{t\in I^+\colon \psi(t)\ge0\}\big|+
\frac1{2|I^-|}\big|\{t\in I^-\colon \psi(t)\ge0\}\big|
\\
&\ge\frac12\Bel(x^+)+\frac12\Bel(x^-)-\eta.
\end{align*}
Since this inequality holds for an arbitrarily small $\eta$, we can pass to the
limit $\eta\to0$, what gives us the required assertion.
\end{proof}

It will be convenient to change variables $x_1=y_1-y_2$,
$x_2=y_1+y_2$, $x_3=y_3$ and introduce a function $\bfM(y)\df\Bel(x)$ defined in the
domain $G\df\{y\in\bR^3\colon |y_1-y_2|\le y_3\}$. Then the main inequality for the function
$\bfM_\pm$ means that it is concave if either $y_1$ is fixed, or $y_2$ is fixed. For the function
$\bfM_\eps$ the condition is more restrictive: it is concave in any direction from the cone
$(y_1^+-y_1^-)(y_2^+-y_2^-)\le0$, since 
$$
(x_2^+-x_2^-)^2-(x_1^+-x_1^-)^2=4(y_1^+-y_1^-)(y_2^+-y_2^-)\,.
$$

\subsection{Supersolution.}
\label{supersol}

\begin{lemma}
\label{sup_sol}
Let $B$ a continuous function on $\Omega$ satisfying the main
inequality~\eqref{mi} and the obstacle condition~\eqref{obst_cond}. Then $\Bel(x)\le B(x)$.
\end{lemma}

\begin{proof}
Let us fix a point $x\in\Omega$ and a pair of admissible functions $\vf$, $\psi$
on $I=[0,1]$ corresponding to $x$, i.\,e., $\bp{\vf,\psi}=x$. Let us introduce a temporary notation
$f\cii J$ for the restriction of the function $f$ on the interval $J$. Using consequently main
inequality for the function $B$ we can write down the following chain of
inequalities
\begin{align*}
B(\bp{\vf,\psi})&\ge\frac12\big(B(\bp{\vf\cii{I^+}\!,\psi\cii{I^+}})+B(\bp{\vf\cii{I^-}\!,\psi\cii{I^-}})\big)
\\
&\ge\sum_{J\in\cD,\,|J|=2^{-n}}\frac1{|J|}B(\bp{\vf\cii J,\psi\cii J})=\int_0^1B(x^{(n)}(t))dt\,,
\end{align*}
where $x^{(n)}(t)=\bp{\vf\cii J,\psi\cii J}$, if $t\in J$, $|J|=2^{-n}$.

Note that without loss of generality we may assume that $B(x)\le1$ and the conclusion would be only stronger. 
Indeed, we can consider the function $\tilde B=\min\{B,1\}$. And this new function satisfies both the concavity condition~\eqref{mi} 
and the obstacle condition~\eqref{obst_cond} with $B(x_1,x_2,|x_1|)=1$ for $x_2\ge0$. Thus, 
since $x^{(n)}(t)\!\to\!(\vf(t),\psi(t),|\vf(t)|)$ almost everywhere (at any common Lebesgue point $t$ 
of the functions $\vf$ and $\psi$), we can pass to the limit in the integral. So, we come to the inequality
\eq[upest]{
B(x)\ge\!\int_0^1\!B(\vf(t),\psi(t),|\vf(t)|)dt
\ge\!\!\!\!\int\limits_{\{t\colon\psi(t)\ge0\}}\!\!\!\!\!\!\!dt\;
=\,\big|\{t\in I\colon\psi(t)\ge0\}\big|\,,
}
where we have used the property $B(x_1,x_2,|x_1|)=1$ for $x_2\ge0$. Now, taking
supremum in~\eqref{upest} over all admissible pairs $\vf$, $\psi$, we get the
required estimate $B(x)\ge\Bel(x)$.
\end{proof}

Now we explain how we will apply this lemma. For a given sequence $\eps=\{\eps\cii J\}$, we denote
$$
T_\eps\vf\df\sum_{J\in\cD(I)}\eps\cii J(\vf, h\cii J) h\cii J(x)\,.
$$
It is a dyadic singular operator (actually, it is a family of operators
enumerated by sequences $\eps$). To prove that this family is uniformly of weak type
$(1,1)$ is the same as to prove
$$
\Bel(x)\le \frac{C\,x_3}{|x_2|}\,.
$$
Indeed, if $\vf$, $\psi$ is an admissible pair corresponding to the point $x$,
then $T_\eps\vf=\psi-x_2$. Therefore, for a given $\vf$ with $\av\vf{}=x_1$ and
$\av{|\vf|}{}=x_3$ the best estimate of the value $|\{t\colon T_\eps\vf\ge\lambda\}|$
gives us the function $\Bel(x)$ with $x_2=-\lambda$. Thus, would we find any
function $B$ with the required estimate and satisfying conditions of
Lemma~\ref{sup_sol} we immediately get the needed weak type $(1,1)$, and
in fact, more precise information on the level set of $T_\eps\vf$.

\subsection{The Bellman function on the boundary}
\label{exact}

First of all we note that the boundary $\pd\Omega$ consists of two independent parts
\begin{align*}
\pd\Omega_+&\df\{x=(x_1,x_2,x_1)\colon x_1\ge0, -\infty<x_2<+\infty\}
\qquad\text{and}
\\
\pd\Omega_-&\df\{x=(x_1,x_2,-x_1)\colon x_1\le0, -\infty<x_2<+\infty\}.
\end{align*}
They are independent in the following sense. If we have a pair of test functions $\vf,\,\psi$
whose Bellman point $x=\bp{\vf,\psi}$ is on the boundary (whence the sign of $\vf(t)$ is 
constant on the whole interval), then after splitting the interval we get a pair of Bellman 
points $x^\pm$ from the same part of the boundary. So, the main inequality~\eqref{mi} has to
be fulfilled separately on $\pd\Omega_+$ and $\pd\Omega_-$. Due to the symmetry condition it
is sufficient to find the function, say, on $\pd\Omega_+$ and further we assume that $x_1\ge0$.

So we look for a minimal function on the half-plane $\{x_1\ge0\}$ satisfying 
the main inequality and the boundary condition~\eqref{bc}. We pass to the 
variable $y$ ($x_1=y_1-y_2$ and $x_2=y_1+y_2$) and look for a function $M$ in the half-plane 
$y_2<y_1$, which satisfies the main inequality (i.\,e. is concave in each variable: in $y_1$, when 
$y_2$ is fixed, and in $y_2$, when $y_1$ is fixed) and with the given values on the boundary
$y_2=y_1$: $\bfM=1$ if $y_1=y_2\ge0$ and $\bfM=0$ if $y_1=y_2<0$.

First, we use concavity of $\bfM$ with respect to $y_2$ for some fixed $y_1\ge0$.
Concave function bounded from below cannot decrease, therefore it has to be identically $1$
on any such ray due to fixed boundary condition. It remains to find $\bfM$ in the domain
$y_2<y_1<0$. Here we use concavity along $y_1$. We know that our function is $0$ at $y_1=y_2$
and, by what we just said, it is $1$ at $y_1=0$, therefore between these two points it is at 
least the linear function $M=1-\frac{y_1}{y_2}$, i.\,e. $\bfM\ge M$, where
$$
M=\begin{cases}
\quad1,&\text{ if }\ y_1\ge0,
\\
1-\frac{y_1}{y_2},&\text{ if }\ y_1<0.
\end{cases}
$$
To prove the opposite inequality we note that $M$ is concave in each variable and it satisfies
the obstacle condition. Therefore, Lemma~\ref{sup_sol} guarantees the required inequality $\bfM_\pm\le M$.
To prove that $\bfM_\eps\le M$ we need to check a bit stronger concavity along any direction from the cone
$(y_1^+-y_1^-)(y_2^+-y_2^-)\le0$. This will be made below when considering the function in the whole domain.

Returning to variable $x$, we can write $\Bel=\frac{2x_1}{x_1-x_2}$ in the half-plane $x_1\ge0$.

As a result we have proved the following
\begin{prop}
\label{bc_no_weight}
\begin{equation}
\label{bc_no_w}
\Bel_\pm(x_1,x_2,|x_1|)=B(x_1,x_2,|x_1|)\df
\begin{cases}
\quad1,&\text{ if }\ x_2\ge-|x_1|,
\\
\frac{2|x_1|}{|x_1|-x_2},&\text{ if }\ x_2\le-|x_1|.
\end{cases}
\end{equation}
\end{prop}

\subsection{Full Bellman function for the weak type estimate}
\label{FBF}

Now we present the full Bellman function:
\begin{theorem}
\label{B_no_weight}
For the function $\Bel$ defined by~\eqref{def_B} we have the following
analytic expression
\begin{equation}
\label{full}
\Bel(x)=B(x)\df
\begin{cases}
\qquad 1,&\text{ if }\ x_3+x_2\ge0, 
\\
\displaystyle
1-\frac{(x_3+x_2)^2}{x_2^2-x_1^2},&\text{ if }\ x_3+x_2<0.
\end{cases}
\end{equation}
\end{theorem}

\begin{proof}
As above we change variables
$$
x_1= y_1-y_2\,,\ x_2=y_1+y_2\,,\ x_3=y_3\,,\quad\text{i.e.}\quad
y_1=\frac{x_1+x_2}2\,,\ y_2=\frac{x_2-x_1}2\,.
$$
and will be looking for a function $M$
$$
M(y)\df B(x)\,,
$$
which is defined in $\Omega\df\{y=(y_1, y_2,y_3): y_3\ge |y_1-y_2|\}$, concave in variables 
$(y_1,y_3)$ and $(y_2,y_3)$, satisfies boundary condition~\eqref{bc_no_w},
or in term of $M$
$$
M(y_1,y_2,|y_1-y_2|)=
\begin{cases}
\qquad1,&\text{ if }\ y_1\ge0\ \text{or}\ y_2\ge0,
\\
\displaystyle
1-\frac{\max\{y_1,y_2\}}{\min\{y_1,y_2\}},&\text{ if }\ y_1<0\ \text{and}\ y_2<0.
\end{cases}
$$

Since the function $\Bel$ is even with respect to $x_1$, as before it is sufficient to consider
the half-space $\{x_1>0\}$, or the half-space $\{y_2<y_1\}$ in $y$-variable. But in fact we
can restrict ourselves to the cone $\{x_2<-x_1<0,\,x_3>x_1\}$ or $\{y_2<y_1<0,\,y_3>y_1-y_2\}$,
because for $y_1>0$ our function is identically $1$ by the same reason as before: it is
concave and bounded by $1$ on every ray $\{y_1=\const, y_2=\const, y_3>y_1-y_2\}$.

The boundary function is not smooth because the boundary itself is not smooth at the line 
$\{x_1=x_3=0\}$ and moreover, the boundary condition on this line has a jump. But inside 
the domain we can look for a smooth candidate $B$. Then it has to satisfy the boundary condition 
$\frac{\pd B}{\pd x_1}|_{x_1=0}=0$, or in terms of $M$
\eq[NbcMT]{
\frac{\pd M}{\pd y_1}\Big|_{y_1=y_2}\!\!=\ \frac{\pd M}{\pd y_2}\Big|_{y_1=y_2}\!.
}

Our function has to be concave in each plane $\{y_1=\const\}$ and in each plane $\{y_2=\const\}$
and we look for a candidate such that its concavity is degenerate in one of these planes, i.\,e. 
in that plane the function $M$ satisfies the \ma equation. Looking on the boundary we 
see that the extremals are segments of the lines $\{y_2=\const\}$ and therefore it is natural 
to look for a solution of the \ma equation 
$$
M_{y_1y_1}M_{y_3y_3}-M_{y_1y_3}^2=0
$$
in this plane. (Section of our domain $\Omega$ by this plane is shown on Figure~\ref{fig1}.)
\begin{figure}
\begin{center}
\begin{picture}(300,200)
\thinlines
\put(180,0){\vector(0,1){180}}
\put(10,20){\vector(1,0){280}}
\put(187,173){\footnotesize $y_3$}
\put(280,8){\footnotesize $y_1$}
\linethickness{.3pt}
\thicklines
\put(140,20){\line(1,1){120}}
\put(140,20){\line(-1,1){80}}
\put(140,20){\circle*{2}}
\put(138,8){\footnotesize $y_2$}
\put(124,-2){\footnotesize $(x_1=0)$}
\multiput(140,20)(0,3){55}{\circle*{1}}
\put(180,60){\circle*{2}}
\put(180,60){\line(-1,1){100}}
\put(187,58){\footnotesize $-y_2$}
\put(240,150){\footnotesize $y_3=|y_1-y_2|$}
\thinlines
\put(210,90){\line(-1,1){90}}
\put(117,170){$\scriptscriptstyle L$}
\put(180,60){\line(-2,1){40}}
\put(180,60){\line(-1,0){40}}
\put(180,60){\line(-2,-1){40}}
\multiput(140,40)(-2,-1){7}{\circle*{1}}
\multiput(140,60)(-2,0){20}{\circle*{1}}
\multiput(140,80)(-2,1){40}{\circle*{1}}
\put(180,33){$\scriptscriptstyle M=1-\frac{y_1}{y_2}$}
\put(178,35){\vector(-1,0){22}}
\put(150,120){$\scriptscriptstyle M=1$}
\put(160,115){\vector(0,-1){33}}
\put(60,43){$\scriptscriptstyle M_{y_1}=M_{y_2}$}
\put(100,45){\vector(1,0){40}}
\end{picture}
\caption{Intersection of the domain $\Omega$ with a plane $y_2=\const$}
\label{fig1}
\end{center}
\end{figure}
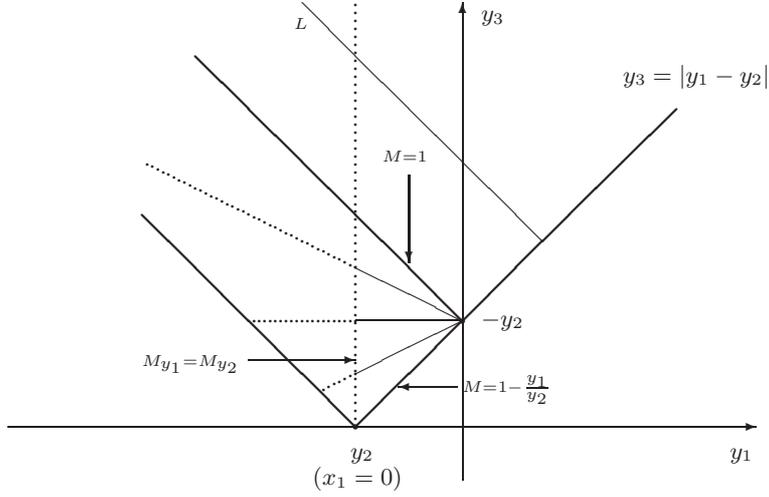
Note that the half-lines $\{y_3+y_1=\const\colon y_3>y_1-y_2\}$ are in the domain if $\const\ge y_2$.
Moreover, if $\const\ge -y_2$ (recall that $y_2<0$), then the boundary value on this ray (the ray $L$ on
Fig.~\ref{fig1}) is $1$, and hence it is identically $1$ for $y_3+y_2+y_1\ge0$, by the same reason 
as before: concave function bounded from below cannot decrease on an infinite interval. 

Therefore we need to solve the \ma equation
only in the triangle with the vertices $(0,y_2,-y_2)$, $(y_2,y_2,0)$, and $(y_2,y_2,-2y_2)$:
$$
\{y=(y_1,y_2,y_3)\colon y_2=\const, y_1>y_2, y_1-y_2<y_3<-y_1-y_2\}
$$
with the boundary conditions
$$
\begin{gathered}
M(y_1,y_2,y_1-y_2)=1-\frac{y_1}{y_2},\qquad M(y_1,y_2,-y_1-y_2)=1,
\\
M_{y_1}(y_2,y_2,y_3)=M_{y_2}(y_2,y_2,y_3)\,.
\end{gathered}
$$

Our function is linear on two sides of the triangle, so the minimal concave function linear on two
sides is the linear function it the whole triangle, however this function does not satisfies the
boundary condition on the side $y_1=y_2$. Therefore, the extremal lines cannot intersect inside the
triangle and the only way to foliate this triangle without singularities inside the domain is
a fan of straight line segments starting from the point $(0,y_2,-y_2)$, which we parametrize
by the slope $k$ of each extremal line:
\eq[extr-linesMT]{
y_3=ky_1-y_2\,.
}
The slope runs over the interval $[-1,1]$. For $k=-1$ we get the upper side of the
triangle $y_1+y_2+y_3=0$ where $M=1$, for $k=1$ we have the lower side $y_3=y_1-y_2$
where $M=1-\frac{y_1}{y_2}$. On all other extremal lines $M$ is linear in $y_1$ as well
$$
M=1+m(k,y_2)y_1
$$
and our task is to find its slope $m=m(k,y_2)$ with the prescribed values at the points $k=\pm1$:
$m(-1,y_2)=0$ and $m(1,y_2)=-\frac1{y_2}$.
We find this function from the boundary condition~\eqref{NbcMT} on the third side of the triangle.

First we deduce from~\eqref{extr-linesMT} that $k=k(y_1,y_2,y_3)=\frac{y_3+y_2}{y_1}$ and hence
$$
\frac{\pd k}{\pd y_1}=-\frac{k}{y_1}\qquad\text{and}\qquad\frac{\pd k}{\pd y_2}=\frac1{y_1}\,.
$$
Therefore,
$$
\begin{gathered}
\frac{\pd M}{\pd y_1}=m+y_1\frac{\pd m}{\pd k}\,\frac{\pd k}{\pd y_1}=m-k\frac{\pd m}{\pd k}\,,
\\
\frac{\pd M}{\pd y_2}=y_1\Big(\frac{\pd m}{\pd y_2}+\frac{\pd m}{\pd k}\,\frac{\pd k}{\pd y_2}\Big)=
y_1\frac{\pd m}{\pd y_2}+\frac{\pd m}{\pd k}\,.\rule{0pt}{25pt}
\end{gathered}
$$
Thus, the boundary condition~\eqref{NbcMT} turns into the following equation
$$
m-(k+1)\frac{\pd m}{\pd k}=y_2\frac{\pd m}{\pd y_2}\,,
$$
which has the general solution of the form
$$
m(k,y_2)=(k+1)\Phi\Big(\frac{k+1}{y_2}\Big)\,,
$$
where $\Phi$ is an arbitrary function. Since $m(1,y_2)=-\frac1{y_2}$, we have $\Phi(t)=-\frac t4$.
And finally
$$
M(y)=1-\frac{(k+1)^2}{4y_2}y_1=1-\frac{(y_1+y_2+y_3)^2}{4y_1y_2}\,,\qquad\text{if}\quad y_1+y_2+y_3<0,
$$
or
$$
B(x)=1-\frac{(x_2+x_3)^2}{x_2^2-x_1^2}\,,\qquad\text{if}\quad x_2+x_3<0.
$$
And our function is identically one on the rest of the domain.

Now it is an easy task to check that the found function $M$ satisfies concavity
conditions from Lemma~\ref{MTtuda}. Since our candidate is $C^1$-smooth function, the
desired concavity is sufficient to check only on the subdomain, where our candidate is less than
one, i.\,e., where $y_1+y_2+y_3<0$. For us there is important that $y_1<0$ and $y_2<0$ on this
part of the domain. We shall check the main inequality (condition~\ref{mi})
in the differential form, namely, we check that the quadratic form of the Hessian of $M$
is not positive in the required directions.
Direct calculations gives the following expression for the Hessian matrix:
$$
\frac{d^2M}{dy^2}=
\begin{pmatrix}
M_{y_1y_1}\! & \!M_{y_1y_2}\! & \!M_{y_1y_3}
\\
M_{y_2y_1}\! & \!M_{y_2y_2}\! & \!M_{y_2y_3}
\\
M_{y_3y_1}\! & \!M_{y_3y_2}\! & \!M_{y_3y_3}
\end{pmatrix}=
\begin{pmatrix}
\ds-\frac{(y_2\!+\!y_3)^2}{2y_1^3y_2} & \ds\frac{y_1^2\!+\!y_2^2\!+\!y_3^2}{4y_1^2y_2^2} 
& \ds\frac{y_2\!+\!y_3}{2y_1^2y_2}
\\
\ds\frac{y_1^2\!+\!y_2^2\!+\!y_3^2}{4y_1^2y_2^2} & \ds-\frac{(y_1\!+\!y_3)^2}{2y_1y_2^3} 
& \ds\frac{y_1\!+\!y_3}{2y_1y_2^2}\rule{0pt}{25pt}
\\
\ds\frac{y_2\!+\!y_3}{2y_1^2y_2} & \ds\frac{y_1\!+\!y_3}{2y_1y_2^2} 
& \ds-\frac1{2y_1y_2}\rule{0pt}{25pt}
\end{pmatrix}
$$
and its quadratic form can be written as follows:
$$
\big(\frac{d^2M}{dy^2}\xi,\xi\big)=
-\frac1{2y_1y_2}\Big(\xi_3-\frac{y_1\!+\!y_3}{y_2}\xi_2-\frac{y_2\!+\!y_3\!}{y_1}\xi_1\Big)^2
+\frac{(y_1\!+\!y_2\!+\!y_3)^2}{2y_1^2y_2^2}\xi_1\xi_2\,.
$$
In our part of the domain we have $y_1<0$ and $y_2<0$, therefore this quadratic form is negative 
if $\xi_1\xi_2\le0$. So, due to Lemma~\ref{sup_sol} we have inequality 
$\Bel_\pm(x)\le\Bel_\eps(x)\le B(x)$.

To prove the theorem we need to check the converse inequality $\Bel(x)\ge B(x)$. 
For this Bellman function it is very easy due to its following special property. 
Note that the function $M$ is linear on the extremal lines not
only in the triangle mentioned above, but also on the continuation of each extremal 
line as well (see Fig.~\ref{fig1}). Indeed, all extremal lines in the triangle under 
investigation are parametrized by their slope $k$, $-1<k<1$, and have the form
$$
y_3=ky_1-y_2,\qquad y_2\le y_1\le0,
$$
and the found function on this line is
$$
M(y_1,y_2,ky_1-y_2)=1-\frac{(k+1)^2}{4y_2}y_1\,.
$$
Thus, we see that this function is linear not only on the interval $y_1\in(y_2,0)$, but
for $y_1<y_2$ as well. So we can continue this extremal line up to its second point of
intersection with the boundary $y_3=|y_2-y_1|$, where this $M$ coincides with $\bfM$. In result
we have two points where the concave function $\bfM$ coincides with the linear function $M$,
therefore between these two points we have $\bfM(y)\ge M(y)$. Since the described continued
extremal line foliate the whole domain $y_1+y_2+y_3<0$, we have the desired inequality
for arbitrary point $y$ from $\Omega$.
\end{proof}

\begin{remark}
We would like to mention that the function~\eqref{full} was published by A.\,Os\c{e}kowski in~\cite{adam}.
It was found by him absolutely independently, but a bit later than the preliminary version of this paper
was accessible in the web (see~\cite{NRVV0}). However we would like to emphasize that in~\cite{adam} not only
this function is presented supplying us with the estimate of the measure where $\{\psi\ge\lambda\}$, 
but the more difficult function giving the estimate for the set $\{|\psi|\ge\lambda\}$ is found as well.
\end{remark}

\subsection{About coincidence of $\Bel_\pm$ with $\Bel_\eps$}

In this subsection we would like to underline that the fact of this coincidence is absolutely
not evident. In many cases as in the famous $L^p$ result of Burkholder the estimation for 
differentially subordinate martingales is the same as for $\pm$-transform. And the natural reason
for this is that any differentially subordinate martingale is a convex combination of $\pm$-transforms.
Indeed, if we fix a martingale $\psi$ being differentially subordinate to $\vf$, i.\,e.
$$
T_\eps\vf\df\psi=\sum_{J\in\cD(I)}\eps\cii J(\vf,h\cii J)h\cii J,\qquad|\eps\cii J|\le1\,,
$$
then every number $\eps\cii J$ can be represented as a convex combination of $\pm1$:
$$
\eps\cii J= \sum_{k=1}^{\infty} 2^{-k} \eps\cii{k,J}\,,\qquad \eps\cii{k,J}=\pm1\,.
$$
Therefore,
$$
T_\eps=\sum_{k=1}^{\infty} 2^{-k} T_{\eps\cii k}\,,
$$
If we were interested in the estimate of $T_\eps$ in a Banach space $X$ (say, $X=L^p$, $p>1$), 
then this  representation would show that
$$
\sup_{\eps\colon\eps\cii J\in[-1,1]}\!\!\|T_\eps\|_X\ =
\!\!\sup_{\eps\colon\eps\cii J\in\{-1,1\}}\!\!\|T_\eps\|_X\,.
$$

However, we are interested in the case $X=L^{1,\infty}$. Here one can use Lemma of Stein and Weiss:
\begin{lemma}
\label{SW}
Let $\{g_j\}$ be a sequence of non-negative measurable functions, such that 
$\|g_j\|_{L^{1, \infty}} \le 1$ for all $j$.
Let $\{c_j\}$ be a sequence of non-negative scalars such that $\sum c_j=1$ and 
$\sum c_j \log\frac1{c_j} =K <\infty$.
Then
$$
\big\|\sum_j c_j g_j \big\|_{L^{1, \infty}} \le 2(K+2)\,.
$$
\end{lemma}

See~\cite{StW} for the proof. From this lemma, we would conclude that
$$
\sup_{\eps\colon\eps_J\in[-1,1]}\!\!\|T_\eps\|_{L^{1,\infty}}\le 
\;2\big(2+\log2\sum_{k=1}^\infty k2^{-k}\big)
\!\!\sup_{\eps\colon\eps_J\in\{-1,1\}}\!\!\|T_\eps\|_{L^{1,\infty}}\,.
$$

However, Theorem \ref{B_no_weight} gives a better result:

\begin{cor}
\label{const1}
$$
\sup_{\eps\colon\eps_J\in[-1,1]}\!\!\|T_\eps\|_{L^{1,\infty}}\ = 
\!\!\sup_{\eps\colon\eps_J\in\{-1,1\}}\!\!\|T_\eps\|_{L^{1,\infty}} \,.
$$
\end{cor}

\section{The Bellman function of weak weighted estimate of 
martingale transform and its properties.}
\label{PropBEL-MT}

Let $\vf$ be a dyadic martingale starting at $x_1$ and $\psi$ is its martingale 
transform, starting at $x_2$, i.\,e.,
$$
\vf=x_1+\!\sum_{J\in\cD(I)}\!\!(\vf, h\cii J) h\cii J\,, \qquad
\psi=x_2+\!\sum_{J\in\cD(I)}\!\!\eps\cii J(\vf,h\cii J)h\cii J\,.
$$
We consider here the following classes of martingale transforms: the martingale $\psi$
is differentially subordinate to $\vf$, i.\,e.  we consider the case when we assume that $|\eps\cii J|\le1$.
The first class of admissible pairs $\{\vf,\,\psi\}$ we denote by $\mfA_\pm$, the second one by
$\mfA_\eps$.

Passing to the weighted case we need to investigate a Bellman function of more variables.
Now two additional variables $x_4$ and $x_5$ appear describing a test weight $w$. 
We put
$$
x_4=\av wI\qquad\text{and}\qquad x_5=\inf_I w\,.
$$
The test weight $w$ will run over the set of all $A_1$ weight with $\Qq\le Q$ and with
the prescribed parameters $x_4$ and $x_5$. This, by the way, means that these parameters
must satisfy the following condition: $x_4\le Qx_5$. 

The coordinates $x_1$ and $x_2$ will be the same, but the coordinate $x_3$ we need 
to change slightly:
$$
x_3=\av{|\vf|w}I,
$$
because now we fix a wighted norm of the test function $\vf\in L^1(I, w\,dx)$.
A Bellman point $x=(x_1,x_2,x_3,x_4,x_5)=\bp{\vf,\psi,w}$ is defined by a dyadic martingale
$\vf$ started at $x_1$, by a subordinated to $\vf$ martingale $\psi$ started at $x_2$, and by
a $A_1$ weight $w$. The Bellman function at this point is defined as follows:
\begin{equation}
\label{BelMT}
\Bel(x)\df\Bel_Q(x)\df
\sup\,\frac1{|I|}w(\{t\in I: \psi(t)\ge0\})\,,
\end{equation}
where the supremum is taken over all admissible triples $\vf,\,\psi,\,w$. We mark the Bellman 
function by the index $Q$ to emphasize that it depends on a fixed parameter $Q$. And in fact 
we are interested just in the dependence of $\Bel_Q$ on this parameter. However during our 
calculations we will omit this index.

This Bellman function is defined in the following subdomain of $\bR^5$:
\begin{equation}
\label{O5}
\Omega:=\{x\in \bR^5\colon x_3\ge |x_1|x_5,\ 0<x_5\le x_4\le Qx_5\}\,.
\end{equation}
Note that formally the Bellman function is defined on the whole $\bR^5$, but in the domain $\Omega$
we include only the points, for which the set of test functions is not empty and therefore 
$\Bel(x)\neq-\infty$ (we would like to assume that $\Bel\ge0$).

\subsection{The properties of $\Bel_{Q}$.}
\label{propMT}

\subsubsection{The first property: boundary conditions.} 

On the boundary $x_4=x_5$ the weight is a constant function $w=x_4=x_5$, and therefore
$$
\Bel_Q(x):= \Bel(x_1,x_2,x_3,x_5,x_5)=
\begin{cases}
\qquad x_5,&\text{ if }\ x_3+x_2x_5\ge0, 
\\
\displaystyle
x_5\Big(1-\frac{(\frac{x_3}{x_5}+x_2)^2}{x_2^2-x_1^2}\Big),&\text{ if }\ x_3+x_2x_5<0.
\end{cases}
$$
As we already mentioned , we will usually skip subscript $Q$ and write simply $\Bel$ instead of $\Bel_Q$.
%The second boundary $x_3=|x_1|x_5$ requires a special consideration.
%\begin{lemma}
%\label{sideBC}
%$$
%\Bel(x_1,x_2,|x_1|x_5,x_4,x_5)=
%\begin{cases}
%\qquad x_4,&\text{ if }\ |x_1|+x_2\ge0, 
%\\
%\displaystyle
%x_4+\frac{|x_1|+x_2}{|x_1|-x_2}x_5,&\text{ if }\ |x_1|+x_2<0,\ \rule{0pt}{20pt}
%x_4\le\frac{(2Q-1)|x_1|-x_2}{|x_1|-x_2}x_5,
%\\
%\displaystyle
%\frac{2Q|x_1|x_5}{|x_1|-x_2},&\text{ if }\ |x_1|+x_2<0,\ \rule{0pt}{20pt}
%x_4>\frac{(2Q-1)|x_1|-x_2}{|x_1|-x_2}x_5.
%\end{cases}
%$$
%\end{lemma}
%
%\begin{proof}
%\end{proof}}

\subsubsection{The second property: the homogeneity.} 
\label{subHomog}
It is clear that if $\{\vf,\psi,w\}$ is the set of admissible triples for a point $x\in\Omega$,
then the set of triples $\{s_1\vf,s_1\psi,s_2w\}$ is admissible for the point 
$$
\tilde x=(s_1x_1,s_1x_2,s_1s_2x_3, s_2x_4,s_2x_5)
$$
for an arbitrary pair of positive numbers
$s_1,\,s_2$. Then by the definition of the Bellman function we have
\eq[homogenMF]{
\Bel(\tilde x)=s_2\Bel(x)\,.
}

In what follows we deal mainly with the restriction $\cB$ of $\Bel$ to the intersection of $\Omega$ with the three-dimensional affine
plane 
\eq[G30]{
%G=\{x\in\Omega\colon x_2=-1,\;x_5=1\}=
\{(x_1,-1, x_3,x_4, 1)\colon|x_1|\le x_3,\;1\le x_4\le Q\}\,,
}
i.\,e. the function
\eq[cB]{
\cB(x_1,x_3,x_4)=\Bel(x_1,-1,x_3,x_4,1)\,.
}
We will identify the above-mentioned part of  the three-dimensional affine
plane 
 with a subdomain of $\bR^3$: 
 \eq[G3]{
G:= \{(x_1, x_3,x_4)\colon|x_1|\le x_3,\;1\le x_4\le Q\}\,.
}

\bigskip

For $x_2\ge0$ we always have $\Bel(x)=x_4$ because for any such point the constant test function
$\psi=x_2$ is admissible, and for $x_2<0$ we can reconstruct $\Bel$ from $\cB$ due to 
homogeneity~\eqref{homogenMF}:
choosing $s_1=-x_2^{-1}$ and $s_2=x_5^{-1}$ we get
\begin{equation}
\label{Bn}
\Bel(x) = x_5 \Bel(-\frac{x_1}{x_2},-1,-\frac{x_3}{x_2x_5},\frac{x_4}{x_5},1)
= x_5 \cB(-\frac{x_1}{x_2},-\frac{x_3}{x_2x_5},\,\frac{x_4}{x_5}\,)\,.
\end{equation}

\subsubsection{The third property:  special form of concavity.}
\label{subMain}
Here we state our main inequality, the weighted analog of Lemma~\ref{MTtuda}.

\begin{lemma}
\label{tudaweight} 
Let $x^\pm$ be two points in $\Omega$ such that $|x^+_2-x^-_2|\le|x^+_1-x^-_1|$ 
and let the point $x$ with $x_i=\frac12(x^+_i+x^-_i)$ for $1\le i\le4$
and $x_5=\min\{x^+_5,x^-_5\}$ be in $\Omega$ as well. Then
\begin{equation}
\label{mi1} \Bel(x)-\frac{\Bel(x^+)+\Bel(x^-)}2\ge 0\,.
\end{equation}
\end{lemma}

\begin{proof}
We repeat almost verbatim the proof of Lemma~\ref{MTtuda}.
Fix $x^\pm\in\Omega$, and take two triples of test functions $\vf^\pm$, $\psi^\pm$, $w^\pm$ 
giving the supremum in $\Bel(x^+)$, $\Bel(x^-)$ respectively up to a small number $\eta>0$. 
Using the fact that the function $\Bel$ does not depend on the interval where the test 
functions are defined, we assume that $\vf^+$, $\psi^+$, $w^+$ live on $I^+$ and $\vf^-$, 
$\psi^-$, $w^-$ live on $I^-$,
i.e.,
$$
\vf^\pm=x^\pm_1+\!\!\!\!\sum_{J\in\cD(I)}\!\!\!\!a\cii J h\cii J\,,\qquad 
\psi^\pm=x^\pm_2+\!\!\!\!\sum_{J\in\cD(I)}\!\!\!\!\eps\cii J a\cii Jh\cii J\,,\quad|\eps\cii J|\le1\,.
$$

Consider
$$
\vf(t):=
\begin{cases}
\vf^+(t)\,,&\text{if }\;t\in I^+
\\
\vf^-(t)\,,&\text{if }\;t\in I^-
\end{cases}
=\frac{x^+_1+x^-_1}2+\frac{x^+_1-x^-_1}2h\cii{I}+ \!\!\!\!\sum_{J\in\cD(I)}\!\!\!\!a\cii Jh\cii J\,,
$$
$$
\psi(t):=
\begin{cases}
\psi^+(t)\,,&\text{if }\;t\in I^+
\\
\psi^-(t)\,,&\text{if }\;t\in I^-
\end{cases}
=\frac{x^+_2+x^-_2}2+\frac{x^+_2-x^-_2}2h\cii{I}+ \!\!\!\!\sum_{J\in\cD(I)}\!\!\!\!\eps\cii Ja\cii Jh\cii J\,.
$$
and 

$$
w(t):=
\begin{cases}
w^+(t)\,,&\text{if }\;t\in I^+,
\\
w^-(t)\,,&\text{if }\;t\in I^-.
\end{cases}
$$

Since $|x^+_2-x^-_2|\le|x^+_1-x^-_1|$ and all $|\eps\cii J|\le1$, $\psi$ is subordinated
to $\vf$. Moreover, according to hypothesis of the Lemma, the point $x$ is in $\Omega$,
whence $x_4\le Qx_5$, i.\,e. $\Qq\le Q$. Therefore the triple $\vf$, $\psi$, $w$ is an 
admissible triple of the test functions corresponding to the point $x$, and
\begin{align*}
\Bel(x)&\ge\frac1{|I|}w\big(\{t\in I_0\colon \psi(t)\ge0\}\big)
\\
&=\frac1{2|I^+|}w^+\big(\{t\in I^+\colon \psi(t)\ge0\}\big)+
\frac1{2|I^-|}w^-\big(\{t\in I^-\colon \psi(t)\ge0\}\big)
\\
&\ge\frac12\Bel(x^+)+\frac12\Bel(x^-)-2\eta.
\end{align*}
Since this inequality holds for an arbitrary small $\eta$, we can pass to the
limit as $\eta\to0$, what gives us the required assertion.
\end{proof}

\subsubsection{The forth property:  $\Bel$ decreases in $x_5$} 
\label{prop4}
This is a corollary of the preceding property, i.\,e. it follows from the main inequality. 
Indeed if we put in the hypotheses of Lemma~\ref{tudaweight} $x_i^+=x_i^-$ for $1\le i\le4$ 
and $x_5^+>x_5^-$, then $x_5=x_5^-$ and inequality~\eqref{mi1} turns into
\begin{equation}
\label{decr_m}
\Bel(x_1,x_2,x_3,x_4,x_5^-)-\Bel(x_1,x_2,x_3,x_4,x_5^+)\ge0\,,
\end{equation}
which means that $B$ is monotone in $x_5$. 

\subsubsection{The fifth property: function $t\mapsto\frac1t\cB(x_1,tx_3, tx_4)$ is increasing}
\label{prop5}
Function $\cB$ was defined in \eqref{cB}. 
This property of $\cB$ is in fact the preceding property rewritten in terms of $\cB$. Indeed, if we put $x_2=-1$
and use~\eqref{Bn} and ~\eqref{decr_m}  we get the required monotonicity (we just rewrite  ~\eqref{Bn} and ~\eqref{decr_m}  and use the notations $t^+=\frac1{x_5^-}$ and $t^-=\frac1{x_5^+}$).

\subsubsection{The sixth property: function $\cB$ is concave}
\label{prop6}
Lemma~\ref{tudaweight} applied to the case $x_2^+=x_2^-=-1$ and $x_5^+=x_5^-=1$ guarantees 
the stated concavity.

\subsubsection{The seventh property: the symmetry  and monotonicity in $x_1$} 
It is easy to see from the definition that $\Bel$, and hence $\cB$ as well, is even 
in its variable $x_1$. 

Concavity of $\cB$ (in $x_1$) and this symmetry together imply 
that $x_1\mapsto \cB(x_1,x_3,x_4)$ is increasing on $[-x_3,0]$ and decreasing on $[0,x_3]$.

\subsection{The goal and the idea of the proof}
\label{goal}
It would be natural now to solve the corresponding boundary value problem for the \ma equation,
to find the function $\Bel$, as it was done in the unweighted case, and then to find the constant
we are interested in:
$$
C(Q)=\sup\big\{\frac{|x_2|\Bel(x)}{x_3}\colon x_2<0,\; x_3\ge|x_1|x_5,\; x_5\le x_4\le Qx_5\big\}\,.
$$
However for now this task is too difficult for us. So, we use the listed properties of $\Bel$
to prove the following estimate from below on function $\cB$.

\begin{theorem}
\label{logQthm}
If $Q\ge4$ then 
\begin{equation}
\label{logQ}
\cB(x_1,x_3,x_4) \ge\frac1{515} Q(\log Q)^{1/3}x_3\,.
\end{equation}
at some point $(x_1,x_3,x_4)\in G$.
\end{theorem}

Now a couple of words about the idea of the proof of Theorem \ref{logQthm}. Ideally we would 
like to find the formula for $\cB$ (and therefore  for $\Bel$ because of \eqref{Bn}).
To proceed we rewrite the third property of $\Bel$ (see subsection~\ref{subMain}) as a PDE 
on $\cB$. Then, using the boundary conditions on $\cB$ on $\pd G$ (the domain $G$ is defined 
in~\eqref{G3}), we may hope to solve this PDE.
Unfortunately there are many roadblocks on this path, starting with the fact that 
the third property of $\Bel$ is not a PDE, it is rather a partial differential inequality 
in discrete form. In the not weighted case we pay no attention to this important fact. We simply assume
the required smoothness of our function to find a smooth candidate. After such a candidate was found
we have proved that it coincides with the required Bellman function. Now we cannot find a candidate
and we will work with the abstractly defined Bellman function whose smoothness is unknown. 
We will write the inequality in discrete form as a pointwise partial differential inequality, 
but for that we will need a subtle result of Aleksandrov.

\subsection{From discrete inequality to differential inequality via Aleksandrov's theorem}
\label{Al}
As it was mentioned in Subsection~\ref{prop6} the function $\cB$ is concave on its domain 
of definition $G$. By the result of Aleksandrov, see Theorem~6.9 of~\cite{EG}, $\cB$ has 
all second derivatives almost everywhere in $G$. Second property (homogeneity) of function $\Bel$ 
(see~\eqref{Bn}) implies that the function $\Bel$ has all second derivatives almost everywhere in $\Omega$.

First, using this fact we rewrite the homogeneity condition (see Subsection~\ref{subHomog}) in the 
following differential form:
\begin{align}
\label{homogen1}
x_1\Bel_{x_1}+x_2\Bel_{x_2}+x_3\Bel_{x_3}&=0\,;
\\
\label{homogen2}
x_3\Bel_{x_3}+x_4\Bel_{x_4}+x_5\Bel_{x_5}&=\Bel\,.
\end{align}
These equalities we have got by differentiating~\eqref{homogenMF} with respect to $s_1$ and 
with respect to $s_2$ and taking the result for $s_1=s_2=1$.

Our second step is to replace the main inequality in discrete form by the inequality in
the form of a pointwise partial differential inequality. Lemma~\ref{tudaweight} implies
that the quadratic form
\eq[mi4]{
\sum_{i,j=1}^4\Bel_{x_ix_j}\xi_i\xi_j
}
is non-positive at almost any interior point of $\Omega$ and for all vectors $\xi\in\bR^4$
such that $|\xi_2|\le|\xi_1|$.

We consider three partial cases of~\eqref{mi4} with $\xi_1=\xi_2$, with $\xi_1=-\xi_2$, and
with $\xi_2=0$. Moreover, to reduce our investigation to consideration of $2\times2$ matrices we
choose some special relation between $\xi_3$ and $\xi_4$. In the first case we consider the quadratic
form on the vector $\xi$ with 
$$
\xi_1=\xi_2=-\delta_1,\quad\xi_3=x_3(\delta_1+\delta_2),\quad\xi_4=x_4\delta_2. 
$$
In the second case we put 
$$
\xi_1=-\xi_2=\delta_1,\quad\xi_3=x_3(\delta_1+\delta_2),\quad\xi_4=x_4\delta_2. 
$$
Then we get two quadratic forms
$$
\sum_{i,j=1}^4\Bel_{x_ix_j}\xi_i\xi_j=\sum_{i,j=1}^2K^\pm_{ij}\delta_i\delta_j\,,
$$
where we denote by $K^\pm$ two $2\times 2$ unpleasant (on the first glance) matrices:
$$
K^\pm=
\begin{pmatrix}
\!\begin{array}{c}
\scriptstyle\Bel_{x_1x_1}\pm\Bel_{x_1x_2}\mp x_3\Bel_{x_1x_3}\pm\Bel_{x_1x_2}
\\
\scriptstyle+\Bel_{x_2x_2}-2x_3\Bel_{x_2x_3}\mp x_3\Bel_{x_1x_3}+x_3^2\Bel_{x_3x_3}
\end{array}\!&\!
\!\begin{array}{c}
\scriptstyle\mp x_3\Bel_{x_1x_3}-x_3\Bel_{x_2x_3}+x_3^2\Bel_{x_3x_3}
\\
\scriptstyle\mp x_4\Bel_{x_1x_4}-x_4\Bel_{x_2x_4}+x_3x_4\Bel_{x_3x_4}
\end{array}\!
\\
\!\begin{array}{c}
\scriptstyle\mp x_3\Bel_{x_1x_3}-x_3\Bel_{x_2x_3}+x_3^2\Bel_{x_3x_3}
\\
\scriptstyle\mp x_4\Bel_{x_1x_4}-x_4\Bel_{x_2x_4}+x_3x_4\Bel_{x_3x_4}
\end{array}\!&\!
\scriptstyle x_3^2\Bel_{x_3x_3}+2x_3x_4\Bel_{x_3x_4}+x_4^2\Bel_{x_4x_4}\rule{0pt}{25pt}
\end{pmatrix}.
$$
These matrices are non-positively defined and their half sum is the following non-positively defined matrix
\eq[main33]{
\begin{pmatrix}
\scriptstyle\Bel_{x_1x_1}\!+\Bel_{x_2x_2}-2x_3\Bel_{x_2x_3}\!+x_3^2\Bel_{x_3x_3}&
\scriptstyle-x_3\Bel_{x_2x_3}\!+x_3^2\Bel_{x_3x_3}\!-x_4\Bel_{x_2x_4}\!+x_3x_4\Bel_{x_3x_4}\!
\\
\!\scriptstyle-x_3\Bel_{x_2x_3}\!+x_3^2\Bel_{x_3x_3}\!-x_4\Bel_{x_2x_4}\!+x_3x_4\Bel_{x_3x_4}&
\scriptstyle x_3^2\Bel_{x_3x_3}\!+2x_3x_4\Bel_{x_3x_4}\!+x_4^2\Bel_{x_4x_4}\rule{0pt}{17pt}
\end{pmatrix}.
}

Before proceed further we rewrite this matrix in terms of the function $\cB$. For this aim we
have to get rid of the derivatives with respect to $x_2$ in this matrix. We are able to do this
by using~\eqref{homogen1}:
\begin{align}
\notag
-x_2\Bel_{x_2x_3}&=\Bel_{x_3}+x_1\Bel_{x_1x_3}+x_3\Bel_{x_3x_3}\,;
\\
\notag
-x_2\Bel_{x_2x_4}&=x_1\Bel_{x_1x_4}+x_3\Bel_{x_3x_4}\,;
\\
\notag
x_2^2\Bel_{x_2x_2}&=2x_1\Bel_{x_1}+2x_3\Bel_{x_3}+x_1^2\Bel_{x_1x_1}+2x_1x_3\Bel_{x_1x_3}+x_3^2\Bel_{x_3x_3}\,.
\end{align}

Using these expressions at the point $x=(x_1,-1,x_3,x_4,1)$ we can rewrite the matrix~\eqref{main33} as follows
\eq[main34]{
\begin{pmatrix}
(1+x_1^2)\cB_{x_1x_1}+2x_1\cB_{x_1}&
-x_1\cB_{x_3}-x_1x_3\cB_{x_1x_3}-x_1x_4\cB_{x_1x_4}
\\
-x_1\cB_{x_3}-x_1x_3\cB_{x_1x_3}-x_1x_4\cB_{x_1x_4}&
x_3^2\cB_{x_3x_3}+2x_3x_4\cB_{x_3x_4}+x_4^2\cB_{x_4x_4}
\end{pmatrix}.
}

Now we consider the matrix $K^0$, that appears if we take $\xi_1=x_1\delta_1$, $\xi_2=0$, 
$\xi_3=x_3\delta_2$, and $\xi_4=x_4\delta_2$ in our quadratic form 
$$
\sum_{i,j=1}^4\Bel_{x_ix_j}\xi_i\xi_j=\sum_{i,j=1}^2K^0_{ij}\delta_i\delta_j\,.
$$
In result we get
$$
K^0=\begin{pmatrix}
x_1^2\Bel_{x_1x_1}&
x_1x_3\Bel_{x_1x_3}+x_1x_4\Bel_{x_1x_4}
\\
x_1x_3\Bel_{x_1x_3}+x_1x_4\Bel_{x_1x_4}&
x_3^2\Bel_{x_3x_3}+2x_3x_4\Bel_{x_3x_4}+x_4^2\Bel_{x_4x_4}
\end{pmatrix}.
$$
The same matrix at the point $x=(x_1,-1,x_3,x_4,1)$ is
\eq[main37]{
\begin{pmatrix}
x_1^2\cB_{x_1x_1}&
x_1x_3\cB_{x_1x_3}+x_1x_4\cB_{x_1x_4}
\\
x_1x_3\cB_{x_1x_3}+x_1x_4\cB_{x_1x_4}&
x_3^2\cB_{x_3x_3}+2x_3x_4\cB_{x_3x_4}+x_4^2\cB_{x_4x_4}
\end{pmatrix}.
}

Taking the sum of~\eqref{main34} and~\eqref{main37} we get the following non-positive matrix
\eq[main38]{
\begin{pmatrix}
(1+2x_1^2)\cB_{x_1x_1}+2x_1\cB_{x_1}&
-x_3\cB_{x_3}
\\
-x_3\cB_{x_3}&
2(x_3^2\cB_{x_3x_3}+2x_3x_4\cB_{x_3x_4}+x_4^2\cB_{x_4x_4})
\end{pmatrix}\le 0\,.
}

\begin{defin} Consider a subdomain of $G$,
\label{G1}
$$
G_1:=\{ (x_1,x_3,x_4)\in G\colon x_3>2|x_1|,\: 2<x_4<Q\}\,.
$$
\end{defin}

\medskip

Fix now $x=(x_1,x_3,x_4)\in G_1$ (now $x$ is a $3$-vector, not a $5$-vector as above, see \eqref{G3}) and a parameter $t\in [1/2,1]$. 
Consider inequality~\eqref{main38} at the point $x^t=(x_1,tx_3, tx_4)$.

Let us introduce a new function $\beta$, which is certain averaging of $\cB$, namely, for any $x\in G_1$ we put
$$
\beta(x)\df 2\int_{1/2}^1 \cB(x^t)\,dt\,.
$$
Notice several simple facts. First of all, as $\cB$ is concave, the differentiation under the integral sign is easily justified, and we get
$$
x_i\beta_{x_i}(x)=2\int_{1/2}^1x^t_i\cB_{x_i}(x^t)\,dt,\qquad 
x_i^2\beta_{x_ix_i}=2\int_{1/2}^1(x^t_i)^2\cB_{x_ix_i}(x^t)\,dt\,.
$$
For every function $F$ on domain $G$ we introduce the notation,
$$
\gamma\cii F(x)\df x_3^2F_{x_3x_3}+2x_3x_4F_{x_3x_4}+x_4^2F_{x_4}\,,
$$
then
\eq[psi-H]{
\gamma\cii \beta(x)= 2\int_{1/2}^1 \gamma\cii \cB (x^t) dt\,.
}
Now integrate~\eqref{main38} on the interval $t\in [1/2,1]$. 
The previous simple observations allow us now to rewrite our reduced concavity condition in the form
\eq[main39]{
\begin{pmatrix}
(1+2x_1^2)\beta_{x_1x_1}+2x_1\beta_{x_1}&
-x_3\beta_{x_3}
\\
-x_3\beta_{x_3}&
2\gamma\cii \beta
\end{pmatrix}\le 0\,.
}

The reader may wonder, why we are so keen to replace~\eqref{main38} by a virtually the same~\eqref{main39}? 
The answer is because we can give a very good  pointwise estimate on $\gamma\cii \beta(x)$, $x\in G_1$.
Unfortunately we cannot give any pointwise  estimate on $\gamma\cii \cB(x)$, $x\in G$. 

Our reduced concavity condition~\eqref{main39} is equivalent to the assertion that $\gamma\cii \beta\le0$ and
the determinant of the matrix in~\eqref{main39} is non-negative, i.\,e.,
\eq[det]{
 -\gamma\cii\beta \cdot [-(1+2x_1^2)\beta_{x_1x_1}-2x_1\beta_{x_1}]\ge x_3^2 \beta_{x_3}^2\,.
}

Let us denote
$$
R\df\sup\frac{\cB(x)}{x_3}\,,\qquad x=(x_1,x_3,x_4)\in G\,.
$$
Our goal formulated in~\eqref{logQ} is to prove $ R\ge cQ(\log Q)^\eps$ (with $c=\frac1{515}$ and $\eps=\frac13$). We are still not too close, but notice 
that automatically $\cB(x)\le Rx_3$, $x\in G$.

\subsection{Logarithmic blow-up}
\label{Logpower}

First we find a pointwise estimate on $\gamma\cii \beta$. 
\begin{lemma}
\label{main1}
If $x=(x_1,x_3,x_4)$ is such that $|x_1|\le\frac14x_3$ and $x_4\ge4$, then
$$
-\gamma\cii \beta(x)\le8R(|x_1|+\frac{x_3}{x_4})\,.
$$
\end{lemma}

\begin{proof}
Consider the following functions
$$
\rho(t)\df\cB(x^t),\quad x\in G_1,\qquad\text{and}\qquad r(t)\df\rho(1)t-\rho(t)
$$
on the interval $[t_0,1]$, where $t_0=\max(\frac{|x_1|}{x_3},\frac1{x_4})$.

Recall that the function $\rho(t)/t$ is increasing (see property five of $\Bel$ in Section~\ref{propMT}).
Therefore, $\rho(t)/t\le\rho(1)$, i.\,e. $r(t)\ge0$. Since $r$ is convex (because $\rho$ is concave) and 
$r(1)=0$, $r$ is a decreasing function on $[t_0,1]$, in particular $r'(1)\le0$. Let us estimate the maximal
value of $r$ in the following way:
\eq[endp]{
r(t_0)<\rho(1) t_0\le Rx_3 t_0<R\big(|x_1|+\frac{x_3}{x_4}\big)\,.
}

Under the hypotheses of the Lemma we have $t_0\le\frac14$, and therefore
\begin{align}
\notag
-&\int_{1/2}^1 \rho''(t)\,dt\le\int_{1/2}^1 r''(t)\,dt\le4\int_{1/2}^1(t-t_0)r''(t)\,dt
\\
\notag
&\qquad\le4\int_{t_0}^1(t-t_0)r''(t)\,dt=4r'(1)(1-t_0)-4r(1)+4r(t_0)\,.
\end{align}
Using estimate~\eqref{endp} and the properties of $r$ ($r'(1)\le0$ and $r(1)=0$) we get
$$
-\int_{1/2}^1 \rho''(t)\,dt\le4R(|x_1| +\frac{x_3}{x_4})\,.
$$

The equality $\gamma\cii \cB(x^t)=t^2\rho''(t)$ implies
$$
-\int_{1/2}^1 \gamma\cii \cB(x^t)\,dt\le4R(|x_1|+\frac{x_3}{x_4})\,.
$$
So, by~\eqref{psi-H} this is the stated in the Lemma estimate.
\end{proof}

\medskip

Now we would like to get an estimate for $\beta_{x_3}$ from below. For this aim we construct a pair of test functions
$\vf$, $\psi$ and a test weight $w$, which supply us with the following estimate for the function $\cB$.
\begin{lemma}
\label{PodpirPrimer}
If $x=(x_1,x_3,x_4)$ is such that $2x_3+x_1\ge1$, then
$$
\cB(x)\ge\frac{2x_4-1}4\,.
$$
\end{lemma}

\begin{proof}
Below $H_I$ stands for the $L^\infty$-normalized Haar function of interval $I$. Let us take the following test functions on the interval $[0,1]$
\begin{align}
\vf&=x_1+x_3H\cii{(0,1)}+(x_3-x_1)H\cii{(0,\frac12)}+(x_3+x_1)H\cii{(\frac12,1)}\,;\notag
\\
\psi&=-1+x_3H\cii{(0,1)}+(x_3-x_1)H\cii{(0,\frac12)}-(x_3+x_1)H\cii{(\frac12,1)}\,;\notag
\\
w&=1+2(x_4-1)\chi\cii{(\frac14,\frac34)}\,.\notag
\end{align}
The Bellman point corresponding to this triple is $(x_1,-1,x_3,x_4,1)$. The function $\psi$ on the interval $(\frac12,\frac34)$
has the value $2x_3+x_1-1$, where the weight $w$ is $2x_4-1$. Therefore, if $2x_3+x_1\ge1$, then by the definition 
$\cB(x)\ge(2x_4-1)/4$.
\end{proof}

\begin{cor}
\label{cor1}
If $x_3+x_1\ge1$, then
$$
\beta(x)\ge\frac{3x_4-2}8\,.
$$
\end{cor}
\begin{proof}
If $x_3+x_1\ge1$, then $2tx_3+x_1\ge1$ for all $t\in[\frac12,1]$. And therefore,
$$
\beta(x)=2\int_{1/2}^1\cB(x^t)\,dt\ge\frac12\int_{1/2}^1(2tx_4-1)dt=\frac{3x_4-2}8\,.
$$
\end{proof}

\begin{cor}
\label{cor2}
If $x_3+x_1\ge1$ and $x_4\ge2$, then
$$
\beta(x)\ge\frac{x_4}4\,.
$$
\end{cor}

\begin{cor}
\label{cor3}
If $x_4\ge2$, then
$$
\beta(x_1,1,x_4)\ge\frac{x_4}4\,.
$$
\end{cor}
\begin{proof}
Since the function $\Bel$ is even in $x_1$, the functions $\cB$ and $\beta$ are even as well.
Therefore without loss of generality we can assume that $x_1\ge0$. Hence for $x_3=1$ the
condition $x_3+x_1\ge1$ holds, and we have the required estimate.
\end{proof}

\begin{cor}
\label{a-exist}
If $x_4\ge2$, then there exists an $a=a(x_4)\in(0,1]$ such that $\beta(0,a,x_4)=\frac{x_4}8$.
\end{cor}

\begin{proof}
Since the function $\beta$ is continuous, the conditions 
$$
\beta(0,0,x_4)=0\quad\text{and}\quad\beta(0,1,x_4)\ge\frac{x_4}4
$$
guarantee the existence of the required $a$.
\end{proof}

\begin{remark}
The function $\beta$ is increasing in $x_3$ because it is positive, concave, and defined on an infinite interval $(0,\infty)$.
Therefore the root $a$ is unique.
\end{remark}

\begin{lemma}
\label{differ}
For any $x\in G$ we have
\eq[beta0]{
\beta(x)\ge\big(1-\frac{2|x_1|}{x_3}\big)\beta(0,x_3,x_4)\,.
}
\end{lemma}
\begin{proof}
Since $\beta$ is even in $x_1$, we can assume $x_1>0$. The stated estimate is 
immediate consequence of the following two facts, $\beta$ is non-negative and concave in $x_1$:
$$
\beta(x)\ge\big(1-\frac{2x_1}{x_3}\big)\beta(0,x_3,x_4)+\frac{2x_1}{x_3}\beta(\frac{x_3}2,x_3,x_4)\,.
$$
\end{proof}

\begin{lemma}
\label{Hx3}
Let $a=a(x_4)$ be the function described in Corollary~\textup{\ref{a-exist}}.
If $x=(x_1,x_3,x_4)$ is such that $4x_1\le x_3\le a$, $2\le x_4\le Q$, then
\eq[beta_x3]{
\beta_{x_3}(x)\ge\max\Big\{\frac{x_4-16Rx_3}{16a},\frac{x_4}8\Big\}\,.
}
\end{lemma}
\begin{proof}
Since $\beta$ is concave with respect to $x_3$, and $\beta_{x_3}\ge0$ for $x_3\in(0,a)$
we can write
$$
a\beta_{x_3}(x)\ge(a-x_3)\beta_{x_3}(x)\ge\beta(x_1,a,x_4)-\beta(x_1,x_3,x_4)\,.
$$
Assuming that $x_1\ge0$ we can use Lemma~\ref{differ}:
$$
\beta(x_1,a,x_4)\ge\big(1-\frac{2x_1}{a}\big)\beta(0,a,x_4)=
\big(1-\frac{2x_1}{a}\big)\frac{x_4}8\ge\frac1{16}x_4\,.
$$
Together with the general estimate $\beta(x)\le Rx_3$ we obtain
$$
\beta_{x_3}(x)\ge\frac{x_4-16Rx_3}{16a}\,.
$$

To get the second inequality we estimate $\beta_{x_3}(a)$:
$$
\beta_{x_3}(a)\ge\frac{\beta(x_1,1,x_4)-\beta(x_1,a,x_4)}{1-a}\ge\beta(x_1,1,x_4)-\beta(x_1,a,x_4)\,.
$$
Now we use Corollary~\ref{cor3} together with the property of $\beta$ to decrease with respect to $x_1$ for $x_1>0$:
$$
\beta(x_1,1,x_4)\ge\frac{x_4}4\qquad\text{and}\qquad\beta(x_1,a,x_4)\le\beta(0,a,x_4)=\frac{x_4}8\,.
$$
In result we get the required estimate:
$$
\beta_{x_3}(a)\ge\frac{x_4}4-\frac{x_4}8=\frac{x_4}8\,.
$$
\end{proof}

Let us denote the function on the right hand side of~\eqref{beta_x3} by $m$. We can rewrite it in the 
following form:
\eq[beta_x3m]{
m(x_3,x_4)=
\begin{cases}
\frac{x_4-16Rx_3}{16a},&\quad\text{if}\quad x_3\le\frac{(1-2a)x_4}{16R}\,;
\\
\frac{x_4}8,&\quad\text{if}\quad x_3\ge\frac{(1-2a)x_4}{16R}\,.
\end{cases}
}

All preparations are made and we are ready to prove Theorem~\ref{logQthm}.

\medskip

\noindent{\it Proof of \bf Theorem~\ref{logQthm}.}
Now we combine Lemmas~\ref{main1} and~\ref{Hx3} to deduce from~\eqref{det} the following inequality
$$
-(1+2x_1^2)\beta_{x_1x_1}-2x_1\beta_{x_1}\ge
\frac{x_3^2(\beta_{x_3})^2}{-\gamma\cii \beta}\ge\frac{x_3^2m^2}{8R(|x_1|+\frac{x_3}{x_4})}\,,
$$
that holds under assumptions $4|x_1|\le x_3\le a\le1$ and $4\ge x_4\le Q$.
Dividing both part of this inequality over $\sqrt{1+2x_1^2}$ we can rewrite it in the form
$$
-\frac\partial{\partial x_1}\Big(\sqrt{1+2x_1^2}\;\beta_{x_1}\Big)
\ge\frac{x_3^2m^2}{8R(|x_1|+\frac{x_3}{x_4})\sqrt{1+2x_1^2}}\,.
$$
Integrating this inequality and taking into account that $\beta$ is even in $x_1$ (i.\,e. $\beta_{x_1}(0,x_3,x_4)=0$) we get
\begin{align*}
-&\sqrt{1+2x_1^2}\;\beta_{x_1}\ge
\frac{x_3^2m^2}{8R}\int_0^{x_1}\frac{dt}{(t+\frac{x_3}{x_4})\sqrt{1+2t^2}}
\\
&\ge\frac{x_3^2m^2}{8R\sqrt{1+2x_1^2}}\int_0^{x_1}\frac{dt}{t+\frac{x_3}{x_4}}=
\frac{x_3^2m^2}{8R\sqrt{1+2x_1^2}}\log\Big(1+\frac{x_4}{x_3}x_1\Big)\,.
\end{align*}
Once more we divide over the square root and integrate in $x_1$:
\begin{align*}
\beta(0,x_3,x_4)&-\beta(x_1,x_3,x_4)\ge
\frac{x_3^2m^2}{8R}\int_0^{x_1}\log\Big(1+\frac{x_4}{x_3}t\Big)\frac{dt}{1+2t^2}
\\
&\ge\frac{x_3^2m^2}{8R(1+2x_1^2)}\int_0^{x_1}\log\Big(1+\frac{x_4}{x_3}t\Big)\,dt
\\
&=\frac{x_3^3m^2}{8Rx_4(1+2x_1^2)}
\Big[\Big(1+\frac{x_4}{x_3}x_1\Big)\log\Big(1+\frac{x_4}{x_3}x_1\Big)-\frac{x_4}{x_3}x_1\Big]
\\
&\ge\frac{x_3^3m^2}{9Rx_4}
\Big[\Big(1+\frac{x_4}{x_3}x_1\Big)\log\Big(1+\frac{x_4}{x_3}x_1\Big)-\frac{x_4}{x_3}x_1\Big]\,.
\end{align*}
In the last estimate we use the restriction $4|x_1|\le x_3\le1$, whence $1+2x_1^2\le\frac98$.

Now we use inequality~\eqref{beta0} from Lemma~\ref{differ} and the general inequality $\beta(x)\le Rx_3$:
$$
\beta(0,x_3,x_4)-\beta(x_1,x_3,x_4)\le\frac{2x_1}{x_3}\beta(0,x_3,x_4)\le2x_1R\,.
$$
Combining with the preceding inequality we come to the following estimate
$$
\frac{x_3^3m^2}{18R^2x_1x_4}
\Big[\Big(1+\frac{x_4}{x_3}x_1\Big)\log\Big(1+\frac{x_4}{x_3}x_1\Big)-\frac{x_4}{x_3}x_1\Big]\le1\,.
$$
Recall that this estimate we obtained in the following domain of variables:
$$
0\le4x_1\le x_3\le a\quad\text{and}\quad 4\le x_4\le Q\,.
$$
Let us now choose the values of this variables. Since the function $t\mapsto\frac{1+t}t\log(1+t)$ monotonously increases,
we get the best possible estimate when take the maximal possible value of $x_1$, i.\,e. $x_1=\frac14x_3$:
$$
\frac{x_3^2m^2}{18R^2x_4}
\Big[\Big(1+\frac{x_4}4\Big)\log\Big(1+\frac{x_4}4\Big)-\frac{x_4}4\Big]\le1\,.
$$
Since the behavior of the function $a(x_4)$ is unknown, we cannot choose the best possible value of $x_4$, we
take the largest value $x_4=Q$:
$$
\frac{x_3^2m^2}{18R^2Q}
\Big[\Big(1+\frac{Q}4\Big)\log\Big(1+\frac{Q}4\Big)-\frac Q4\Big]\le1\,,
$$
where, of course, $a=a(Q)$ and $m=m(x_3,Q)$. To simplify this expression we use the following elementary estimate:
$$
\big(1+\frac t4\big)\log\big(1+\frac t4\big)-\frac t4\ge\frac t{16}\log t\quad\text{for}\quad t\ge4\,.
$$

To check this inequality we consider the function
$$
f(t)\df16\big(1+\frac t4\big)\log\big(1+\frac t4\big)-4t-t\log t
$$
and check that $f(t)\ge0$ for $t\ge4$.
\begin{align*}
f(4)&=32\log2-16-4\log4=8\log\frac8{e^2}>0\,;
\\
f'(t)&=4\log\big(1+\frac t4\big)-\log t-1\,;
\\
f'(4)&=4\log2-\log4-1=\log\frac4e>0\,;
\\
f''(t)&=\frac4{t+4}-\frac1t=\frac{3t-4}{t(t+4)}>0\quad\text{for}\quad t\ge4\,.
\end{align*}

In result we get
\eq[final]{
\frac{x_3^2m^2}{288R^2}\log Q\le1\qquad\text{for any}\quad x_3\in[0,a]\,.
}
Now we need to investigate the function $x_3\mapsto x_3m(x_3,Q)$ on the interval $[0,a]$.
If $a\ge\frac14$ then this function is increasing and takes its maximal value at the point $x_3=a$, and~\eqref{final}
yields
$$
\frac{a^2Q^2}{288R^2\cdot8^2}\log Q\le1\,.
$$
or
\eq[a1]{
R\ge\frac a{96\sqrt2}Q\big(\log Q\big)^{1/2}\ge
\frac{(\log 4\big)^{1/6}}{4\cdot96\sqrt2}Q\big(\log Q\big)^{1/3}\ge
\frac1{515}Q\big(\log Q\big)^{1/3}\,.
}
We specially make the exponent of logarithm worth ($\frac13$ instead of $\frac12$), because we can get only such
exponent for other values of the unknown parameter $a$.

From now on we assume that $a<\frac14$. In this case the function has a local maximum at the point $x_3=\frac{Q}{32R}$.
Indeed, since $aR\ge\beta(0,a,Q)=\frac Q8$, we have $a\ge\frac Q{8R}>\frac Q{32R}$, therefore the point $x_3=\frac{Q}{32R}$
is in the domain. The value of the function $x_3m(x_3,Q)$ at this the point 
is $\frac{Q}{32R}\cdot\frac{Q}{32a}$. On the other hand at the end of the interval for $x_3=a$ we have the value
$am(a,Q)\ge\frac{aQ}8$. If $a^2<\frac{Q}{128R}$ then we use the first estimate:
$$
1\ge\Big(\frac{Q}{32R}\cdot\frac{Q}{32a}\Big)^2\frac1{288R^2}\log Q
\ge\frac{Q^4}{9\cdot2^{25}R^4}\cdot\frac{2^7R}Q\log Q\ge\Big(\frac{Q}{134R}\Big)^3\log Q\,,
$$
or
$$
R\ge\frac1{134}Q\big(\log Q\big)^{1/3}\,.
$$
In the case if $a^2\ge\frac{Q}{128R}$ we use the second estimate:
$$
1\ge\Big(\frac{aQ}8\Big)^2\frac1{288R^2}\log Q
\ge\frac{Q^3}{9\cdot2^{18}R^3}\log Q
\ge\Big(\frac{Q}{134R}\Big)^3\log Q\,,
$$
and again
\eq[a3]{
R\ge\frac1{134}Q\big(\log Q\big)^{1/3}\,.
}
Therefore, if $a<\frac14$ estimate~\ref{a3} holds.

Comparing the estimates we got for different possible values of the unknown parameter $a$, namely, \eqref{a1} 
and~\eqref{a3} we see that the estimate
$$
R\ge\frac1{515}Q\big(\log Q\big)^{1/3}
$$
is true in all cases. This completes the proof of Theorem~\ref{logQthm}, and therefore the proof of 
Theorem~\ref{weakMT-t}.
\bigskip

\noindent{\bf Acknowledgment.} This paper was prepared mainly during the stay of V.~Vasyunin and A.~Volberg
in the Mathematisches Forschungsinstitut Oberwolfach (the program ``Research in Pairs''). We are very grateful 
for the opportunity to work there and for the remarkable work conditions at the Institute.

\end{document}